\documentclass[english,11pt]{amsart}

\usepackage[english]{babel}
\usepackage[T1]{fontenc}
\usepackage{lmodern}

\usepackage{MnSymbol}
\usepackage[centering]{geometry}

\theoremstyle{plain}
\newtheorem{theo}{Theorem}
\newtheorem{lemm}[theo]{Lemma}

\theoremstyle{definition}

\theoremstyle{remark}
\newtheorem{rema}[theo]{Remark}

\newtheorem{theoA}{Theorem}

\newtheorem{lemmA}[theoA]{Lemma}

\usepackage{microtype}
\usepackage{enumerate}
\usepackage{graphicx}
\usepackage{color}
\usepackage{bm}
\usepackage{mathtools}
\usepackage[dvipsnames]{xcolor}
\definecolor{FlatRed}{RGB}{231,76,60}
\definecolor{FlatGreen}{RGB}{46,204,113}
\definecolor{FlatBlue}{RGB}{52,152,219}
\definecolor{FlatYellow}{RGB}{241,196,15}
\colorlet{FlatViolet}{FlatRed!50!FlatBlue}
\colorlet{FlatBrown}{FlatRed!50!FlatGreen}
\colorlet{FlatOrange}{FlatRed!50!FlatYellow}
\colorlet{FlatCyan}{FlatGreen!50!FlatBlue}

\usepackage{enumitem}
\usepackage{textcomp}
\usepackage[normalem]{ulem}

\usepackage{hyperref}
\hypersetup{
    colorlinks,
    linkcolor={FlatRed},
    citecolor={FlatGreen},
    urlcolor={FlatBlue}
}

\title[Gaussian correlation inequality via supersymmetry]{Gaussian correlation inequalities from supersymmetric dimensional reduction}

\author{Yichao Huang}
\address{Beijing Institute of Technology, School of Mathematics and Statistics, Beijing, China}
\email{yichao.huang@outlook.com}

\begin{document}

\begin{abstract}
We revisit Royen's proof of the Gaussian correlation inequality from a supersymmetric point of view. Many key elements in Royen's proof of this inequality have natural geometric interpretations in terms of supersymmetric dimensional reduction from $\mathbb{R}^{3|2}$ to $\mathbb{R}^{1|0}$. In particular, the auxiliary multivariate Gamma distributions appearing in Royen's Laplace-transform argument arise naturally as the body of a supersymmetric radial variable on $\mathbb{R}^{3|2}$. The generalization to the half-integer multivariate Gamma case also follows naturally as a dimensional reduction from $\mathbb{R}^{k+2|2}$ to $\mathbb{R}^{k|0}$, and we derive an exact interpolation identity for general smooth test functions. This provides an example in which the supersymmetric localization method is applied to prove correlation inequalities with continuous interpolation.
\end{abstract}

\maketitle

\section{Introduction}
The Gaussian correlation inequality was conjectured in the 1950s and was proved by Royen in 2014~\cite{Royen2014}. It went unnoticed until Lata\l a and Matlak~\cite{Lata_a_2017} reorganized the proof, and various expositions started to circulate, including in the Bourbaki seminar~\cite{Barthe2019}. The following equivalent formulation (\cite[Theorem~1.2]{Barthe2019} or~\cite[Theorem~2]{Lata_a_2017}) of the Gaussian correlation inequality is used in Royen's proof:
\begin{theoA}[Gaussian correlation inequality]\label{th:main}
Let $n\geq n_1\geq 1$ be integers, and $X=(X_1,\dots,X_n)\in\mathbb{R}^{n}$ a centered Gaussian vector. Then
\begin{equation*}
    \mathbb{P}\left[\max_{1\leq i\leq n}|X_i|\leq 1\right]\geq\mathbb{P}\left[\max_{1\leq i\leq n_1}|X_i|\leq 1\right]\mathbb{P}\left[\max_{n_1<i\leq n}|X_i|\leq 1\right].
\end{equation*}
\end{theoA}

Although Royen's proof of the Gaussian correlation inequality is elementary and the expositions~\cite{Lata_a_2017,Barthe2019} make it accessible to the probability and statistics communities, several of its key steps remain somewhat mysterious and rely on a series of elegantly combined tricks. The purpose of the present paper is to provide a geometric reinterpretation of these tricks using language and intuition from mathematical physics, in particular supersymmetric calculus. The key idea is to recognize Royen's proof as a \emph{supersymmetric dimensional reduction}. More precisely, we show that the auxiliary multivariate Gamma laws appearing in Royen's proof arise as the bosonic body of a natural super-radius on $\mathbb R^{3|2}$, and that the positivity of the interpolation derivative is the consequence of a Ward identity together with the positivity of bosonic boundary integrals. This gives a precise geometric explanation of the determinant expansions and Laplace-transform identities in Royen's proof.

As a direct consequence of the supersymmetric method, we upgrade the Gaussian (as well as the multivariate Gamma distribution following the notation of~\cite{Barthe2019}) correlation inequality to a stronger result. The indicator function in Theorem~\ref{th:main} can be extended to much more general functionals satisfying a sign condition on its mixed partial derivatives of order at most $1$ in each coordinate:
\begin{theo}\label{th:better}
Let $n,k>0$ be positive integers. Let $C\geq 0$ be a positive semidefinite covariance matrix. Denote by $\Gamma_n(\frac{k}{2},C)$ the (half-integer) multivariate Gamma vector whose law is characterized by the Laplace transform
\begin{equation*}
    \forall \lambda\in\mathbb{R}_{\geq 0}^{n},\quad \mathbb{E}\left[e^{-\langle\lambda, \Gamma_n(\frac{k}{2},C)\rangle}\right]=\det(I_n+\Lambda C)^{-\frac{k}{2}}
\end{equation*}
where $\Lambda=\mathrm{diag}(\lambda_1,\dots,\lambda_n)$.

Let $n=n_1+n_2$ and consider the interpolation matrix
\begin{equation*}
    C(\tau)=\begin{pmatrix} C_{11} & \tau C_{12} \\ \tau C_{21} & C_{22}\end{pmatrix}
\end{equation*}
so that $C=C(1)$ and $C(0)$ is a two-component block-diagonal matrix. Then for any smooth test function $\Phi$ with compact support or sufficient decay,
\begin{equation*}
    \frac{d}{d\tau}\mathbb{E}\left[\Phi\left(\Gamma_{n}(\frac{k}{2},C(\tau))\right)\right]=\frac{k}{2}\sum_{\emptyset\neq J\subset[1,\dots,n]}a_J(\tau)\mathbb{E}\left[(-1)^{|J|}\partial_{J}\Phi\left(\Gamma_{n}(\frac{k}{2}+1,C(\tau))\right)\right]
\end{equation*}
where $a_J(\tau)\geq 0$ is non-negative for all $J\subset[1,\dots,n]$ with an explicit expression (see Lemma~\ref{lemm:AirJordan}), and $\partial_J$ denotes the mixed partial deriative $\partial_J=\prod_{j\in J}\frac{\partial}{\partial x_j}$ of order at most $1$ in each coordinate.
\end{theo}

The Gaussian case corresponds to $k=1$. Furthermore, taking a sequence of suitable smooth approximations of the function $\prod_{j\in[1,\dots,n]}\mathbf{1}_{\{|x_j|\leq \frac{1}{2}\}}$, we recover Theorem~\ref{th:main}.

\subsection{Organization of the paper.}
To make the paper pedagogical, we organize the content roughly according to the order in which we tested and unveiled the supersymmetric interpretation behind the Gaussian correlation inequality. We hope that in this way, the reader can get a feeling about our discovery process, and form some intuition on the type of problems that might have a hidden supersymmetric interpretation. In Section~\ref{sec:notations_and_background}, we provide some minimal background on probability and supersymmetric calculus. In Section~\ref{sec:snippets_of_royen_s_proof}, we focus on several instances of Royen's proof which led us to believe that a supersymmetric reinterpretation is plausible. In Section~\ref{sec:supersymmetric_reinterpretation_of_royen_s_proof}, we explain how Royen's proof of the Gaussian correlation inequality is an instance of the supersymmetric dimensional reduction procedure. In Section~\ref{sec:the_case_of_the_multivariate_gamma_distributions}, we show that in this supersymmetric framework, the extension to half-integer multivariate Gamma random variables is also natural and yields the stronger correlation inequality of Theorem~\ref{th:better}.

Royen's original proof is given in~\cite{Royen2014}. We mainly follow the notation of the expository article~\cite{Lata_a_2017} (see also a French version~\cite{Barthe2019} in the Bourbaki seminar). A different proof of the Gaussian correlation inequality (without a continuous monotonic interpolation) via inverse Brascamp-Lieb is~\cite{Milman_25} (it would be interesting to investigate whether that proof also admits a supersymmetric interpretation!). A modern introduction to the Berezin calculus can be found in~\cite{Witten_2019}. In a broader context, this work lies within the tradition of connecting supersymmetry formalism with probability theory~\cite{Bismut_1986}. The supersymmetric variational method for establishing convex correlation inequalities with continuous parameters was initiated in~\cite{Huang:2026aa} in a hyperbolic model with trivial partition function. By contrast, the present work treats a Gaussian model and directly handles the variation of a non-trivial partition function. We believe that this new supergeometric point of view on correlation inequalities may provide useful clues for related challenges, and the Gaussian Correlation Inequality provides an elementary and enlightening introduction to this approach.

\subsection*{Acknowledgement}
Y.H. is partially supported by the National Key R\&D Program of China (No. 2022YFA1006300) as well as NSFC-12301164, NSCF-12671177. We thank Xiaolin Zeng for helpful exchanges on related projects.

\section{Notations and background}\label{sec:notations_and_background}
\subsection{Gaussian vectors}
We follow the notations of Lata\l a--Matlak~\cite{Lata_a_2017}. Consider the $n$-dimensional centered Gaussian vector $X$ as in the statement of Theorem~\ref{th:main}. Let $n=n_1+n_2$ and denote the covariance matrix of $X$ by blocks as
\begin{equation*}
    C=\begin{pmatrix} C_{11} & C_{12} \\ C_{21} & C_{22}\end{pmatrix}
\end{equation*}
where $C_{ij}$ is of size $n_i\times n_j$. Consider the interpolation covariance matrix
\begin{equation}\label{eq:Interpolate_Cov}
    C(\tau)=\begin{pmatrix} C_{11} & \tau C_{12} \\ \tau C_{21} & C_{22}\end{pmatrix}
\end{equation}
for $0\leq \tau\leq 1$: it continuously decorrelates $X$ into two independent components.

\begin{rema}
Throughout the proof below we will assume that $C$ is positive definite, so that we can use its inverse $C^{-1}$. The general positive semidefinite case follows by approximating $C$ by $C+\epsilon I_n$ and then letting $\epsilon$ go to $0$. Notice also that $C(\tau)$ is positive definite for $0\leq \tau\leq 1$ when $C$ is.
\end{rema}

\subsection{Supersymmetric calculus}
We denote by Latin letters $x,y,z,\dots$ the bosonic variables, and by Greek letters $\xi,\eta,\dots$ the fermionic or Grassmann variables. Grassmann variables anti-commute with each other, e.g., $\xi\eta=-\eta\xi$. The (left) fermionic derivative acts by the rule
\begin{equation*}
    \frac{\partial}{\partial\xi}(\xi F)=F,\quad \frac{\partial}{\partial\xi}F=0
\end{equation*}
for any function $F$ that does not contain $\xi$, and analogously for any other Grassmann variables. The fermionic integral is the same operation as the fermionic derivative:
\begin{equation*}
    \int \partial_\xi F(\xi)=\frac{\partial}{\partial\xi}F(\xi).
\end{equation*}

One can formally Taylor expand any function acting on even (i.e. commuting) variables. For example, one has the following Taylor expansion:
\begin{equation}\label{eq:BasicTaylor}
    F(x+\xi\eta)=F(x)+F'(x)\xi\eta,
\end{equation}
and higher-order terms vanish due to anti-commutation of the fermionic variables. One should think of $x+\xi\eta$ as infinitesimally close to $x$~\cite[Section~2.1.2]{Witten_2019}.

We also recall the fermionic Gaussian integral formula: for any $N\times N$ matrix $\Sigma$,
\begin{equation}\label{eq:Fermionic_GaussianIntegral}
    \int \prod_{i=1}^{N}\partial_{\xi_i}\partial_{\eta_i}e^{-\xi^{t}\Sigma\eta}=\det(\Sigma),
\end{equation}
where $\xi=(\xi_1,\dots,\xi_N)$ and similarly for $\eta$.

\section{Snippets of Royen's proof}\label{sec:snippets_of_royen_s_proof}
To get an idea of the hidden supersymmetric structure behind Royen's proof, we start by re-examining some central properties in his proof with supersymmetric calculation.

\subsection{A positivity property}
In the course of proving the Gaussian correlation inequality, the positivity of the derivative of a certain type of determinant was crucially used. In the notation of~\cite{Lata_a_2017} (see also~\cite[Section~3.2]{Barthe2019}), the statement is the following:
\begin{lemmA}\label{lemm:AirJordan}
For any $\emptyset\neq J\subset\{1,\dots,n\}$ and $\tau\in[0,1]$, let $C(\tau)_J$ be the submatrix of the interpolated covariance kernel~\eqref{eq:Interpolate_Cov} restricted to the subset $J\subset \{1,\dots,n\}$. Then for any such $J$ and $\tau\in[0,1]$,
\begin{equation*}
    a_J(\tau)\coloneqq -\frac{\partial}{\partial\tau}\det(C(\tau)_J)\geq 0.
\end{equation*}
\end{lemmA}
The original proof of this result relies on the Schur complement formula: this is actually the starting point of this paper, since the Schur complement formula is also used to establish formulas for the Berezin superdeterminant, see~\cite[Equation~(3.8)]{Witten_2019}. Therefore, we suspected that there are some hidden supersymmetries behind (at least this segment of) Royen's proof. We now give an alternative supersymmetric proof of the above lemma, inspired by the calculations in~\cite[Lemma~5]{MR3904155}.
\begin{proof}
The first step is to rewrite the determinant as a fermionic integral using~\eqref{eq:Fermionic_GaussianIntegral}. Let $J_1=J\cap\{1,\dots,n_1\}$ and $J_2=J\cap\{n_1+1,\dots,n\}$ so that $J=J_1\cup J_2$. Write
\begin{equation*}
    C(\tau)_J=\begin{pmatrix} C_{11} & \tau C_{12} \\ \tau C_{21} & C_{22}\end{pmatrix}
\end{equation*}
in this decomposition, where we drop the $J$ subscript on the right-hand side. Define the action for the fermionic vectors $\xi,\eta$ indexed by $J$ as
\begin{equation*}
    S(\xi,\eta)=\sum_{i,j\in J}\xi_i C_{ij}\eta_j=(\xi_1)^{t}C_{11}\eta_1+(\xi_2)^{t}C_{22}\eta_2+\tau(\xi_1)^{t}C_{12}\eta_2+\tau(\xi_2)^{t}C_{21}\eta_1,
\end{equation*}
with $\xi=(\xi_1,\xi_2)$ in the decomposition $J=J_1\cup J_2$ and similarly for $\eta=(\eta_1,\eta_2)$.

Let $J$ be the index of a system of interacting particles: in the decomposition $J=J_1\cup J_2$, think of $J_1$ as representing the subsystem that we are interested in, and $J_2$ as representing the outside boundary conditions. The Schur complement formula is equivalent to studying the marginal law on $J_1$: this can be seen by completing the square. We thus determine the effective action $S_{\text{eff}}$ on $(\xi_1,\eta_1)$ (i.e. study its marginal law in probabilistic terms), by integrating out $(\xi_2,\eta_2)$ conditional on $(\xi_1,\eta_1)$. Define
\begin{equation*}
    \xi'_2=\xi_2+\tau C_{22}^{-1}C_{21}\xi_1,\quad \eta'_2=\eta_2+\tau C_{22}^{-1}C_{21}\eta_1,
\end{equation*}
so that $S(\xi,\eta)=S_{\text{eff}}(\xi_1,\eta_1)+S'(\xi'_2,\eta'_2)$ where
\begin{equation*}
    S_{\text{eff}}(\xi_1,\eta_1)=(\xi_1)^{t}(C_{11}-\tau^2C_{12}C_{22}^{-1}C_{21})\eta_1
\end{equation*}
and $S'(\xi'_2,\eta'_2)=(\xi'_2)^{t}C_{22}\eta'_2$ decouples from $(\xi_1,\eta_1)$.

Performing the fermionic Gaussian integral~\eqref{eq:Fermionic_GaussianIntegral} in the decoupled coordinate system $(\xi_1,\eta_1,\xi'_2,\eta'_2)$ yields
\begin{equation*}
    \det(C(\tau)_J)=\det(C_{22})\int \prod_{i\in J_1}\partial_{\xi_i}\partial_{\eta_i} e^{-S_{\text{eff}}(\xi_1,\eta_1)}
\end{equation*}
so that differentiating with respect to $\tau$ inside the integrals gives
\begin{equation*}
    a_J(\tau)=\det(C_{22})\int \prod_{i\in J_1}\partial_{\xi_i}\partial_{\eta_i} e^{-S_{\text{eff}}(\xi_1,\eta_1)} (-2\tau(\xi_1)^{t}C_{12}C_{22}^{-1}C_{21}\eta_1).
\end{equation*}
It remains to establish the positivity of the integral, which can be rewritten with the normalized fermionic expectation $\langle\cdot\rangle$ with respect to the action $S_{\text{eff}}$ as
\begin{equation*}
    a_J(\tau)=-2\tau\det(C(\tau)_{J})\langle (\xi_1)^{t}C_{12}C_{22}^{-1}C_{21}\eta_1\rangle.
\end{equation*}
For this, we use the switching lemma~\cite[Lemma~7]{Huang:2026aa} to trade the fermionic expectation $\langle (\xi_1)^{t}C_{12}C_{22}^{-1}C_{21}\eta_1\rangle$ for a bosonic one $-\langle (x_1)^tC_{12}C_{22}^{-1}C_{21}x_1\rangle$ (here $x_1$ is the vector indexed by $J_1$) with action
\begin{equation*}
    S_1=\frac{1}{2}\sum_{i,j\in J_1}x_i(C_{11}-\tau^2C_{12}C_{22}^{-1}C_{21})x_j.
\end{equation*}
We will recall the proof of the switching lemma in Lemma~\ref{lemm:BosonFermion} when additional notations are introduced. Finally, $a_J(\tau)$ is a positive linear combination of bosonic expectations $\langle (x_1)^tC_{12}C_{22}^{-1}C_{21}x_1\rangle$, which are non-negative as $(x_1)^tC_{12}C_{22}^{-1}C_{21}x_1=|C_{22}^{-\frac{1}{2}}C_{21}x_1|^2\geq 0$.
\end{proof}

\subsection{A determinantal expansion formula}\label{sub:a_determinantal_expansion_formula}
Another curious trick in Royen's proof is the following determinant expansion formula~\cite[Lemma~5.(i)]{Lata_a_2017}: for any square matrix $A$ of size $n$,
\begin{equation}\label{eq:Leibniz}
    \det(I_n+A)=1+\sum_{\emptyset\neq J\subset\{1,\dots,n\}}\det(A_J).
\end{equation}
Proving this is relatively easy, but this identity is quite important in Royen's proof, and we want to understand why it is useful. Investigating this question ultimately leads to the supersymmetric dimensional reduction picture, which we will explain in Section~\ref{sec:supersymmetric_reinterpretation_of_royen_s_proof}. Here, we give a rather trivial proof with fermionic integration as a preparation (several other determinantal identities of Royen's proof can be shown similarly, but we explain below why we believe that understanding~\eqref{eq:Leibniz} is instrumental to the supersymmetric interpretation).

Using~\eqref{eq:Fermionic_GaussianIntegral}, we write
\begin{equation*}
    \det(I_n+A)=\int \prod_{i=1}^{n}\partial_{\xi_i}\partial_{\eta_i}\exp\left(-\sum_{l=1}^{n}\xi_l\eta_l\right)\exp\left(-\sum_{1\leq i,j\leq n}\xi_i A_{ij}\eta_j\right).
\end{equation*}
Taylor expanding the first exponential:
\begin{equation*}
    \exp\left(-\sum_{l=1}^{n}\xi_l\eta_l\right)=\sum_{J'\subset\{1,\dots,n\}}\prod_{l\in J'}(-\xi_l\eta_l).
\end{equation*}
For a fixed $J'$, denote by $J=\{1,\dots,n\}\setminus J'$. Since fermionic variables are nilpotent, in order to pair with the product $\prod_{l\in J'}(-\xi_l\eta_l)$, the second exponential $\exp\left(-\sum_{1\leq i,j\leq n}\xi_i A_{ij}\eta_j\right)$ is reduced to
\begin{equation*}
    \exp\left(-\sum_{i,j\in J}\xi_i A_{ij}\eta_j\right).
\end{equation*}
Using~\eqref{eq:Fermionic_GaussianIntegral} again to integrate out the above display yields~\eqref{eq:Leibniz}.

\begin{rema}
The use of identity~\eqref{eq:Leibniz} in Royen's proof is very audacious: it splits a determinant into $2^{n}$ pieces, and each of these pieces was later shown to contribute positively. If one examines Royen's proof carefully, each piece is ultimately mapped to a lower dimensional simplex of the unit cube $[0,1]^{n}$, but the dimensions of these simplices take different values in $\{0,\dots,n\}$. In the supersymmetric method presented in the sequel, these $2^{n}$ boundary terms are packaged into one single Taylor expansion, giving a geometric explanation of the naturalness of this operation.
\end{rema}

\section{Supersymmetric reinterpretation of Royen's proof}\label{sec:supersymmetric_reinterpretation_of_royen_s_proof}
We now rewrite Royen's proof using a systematic supergeometric interpretation. The main goal is to provide a perhaps physically more natural alternative to the Laplace transformation identification in Royen's proof, where the special probability distributions constructed by Royen appear naturally as canonical objects in the supersymmetric interpretation. In particular, the variational nature of Royen's proof, reinterpreted via supersymmetric integration by parts, remains unaltered.

\subsection{Supersymmetric radius and hypercube}
Recall that in Theorem~\ref{th:main} we are interested in the joint distribution of the radius $|X_i|^2$ for a centered Gaussian vector $X$. The corresponding bosonic action for $X\in\mathbb{R}^{n}=(\mathbb{R}^{1|0})^{n}$ is
\begin{equation}\label{eq:Action_X}
    S_{\text{b}}(X)=\frac{1}{2}\sum_{1\leq i,j\leq n}X_i(C(\tau)^{-1})_{ij}X_j.
\end{equation}
We first use supersymmetric dimensional reduction (but in the opposite direction) to lift $|X_i|^2$ to a super-radius in the superspace $\mathbb{R}^{3|2}$, introducing identical copies $Y,Z$ of $X$ and dual fermionic counterparts $\xi,\eta$ by the total action
\begin{equation*}
    S_{\text{tot}}(X,Y,Z,\xi,\eta)=S_{\text{b}}(X)+S_{\text{b}}(Y)+S_{\text{b}}(Z)+S_{\text{f}}(\xi,\eta)=S_{\text{b}}(X,Y,Z)+\sum_{1\leq i,j\leq n}\xi_i (C(\tau)^{-1})_{ij}\eta_j,
\end{equation*}
and define the super-radius $R_i^2=X_i^2+Y_i^2+Z_i^2+2\xi_i\eta_i$. The Berezin volume form on $(\mathbb{R}^{3|2})^{n}$ is the flat one, $d\mu=\prod_{i=1}^{n}dX_idY_idZ_i\partial_{\xi_i}\partial_{\eta_i}$.\footnote{In the sequel, all one-dimensional bosonic Lebesgue measures are normalized by $(2\pi)^{-1/2}$ so that we omit the renormalization factor $(2\pi)^{-\frac{3n}{2}}$ in $d\mu$.} The super-hypercube of $\mathbb{R}^{3n|2n}$ is defined as
\begin{equation*}
    U_n=\bigcap_{i=1}^{n}\{R_i^2\leq 1\}.
\end{equation*}

Following~\cite[Section~3.4]{Witten_2019}, we need to recall what this means since $R_i$ is not a real number. For simplicity, let us define $\{R_1^2\leq 1\}$; the other cases follow similarly and one then takes the intersection to get $U_n$. In this case, the boundary $\{R_1^2=1\}$ is defined by the equation $f(R_1)=0$ where $f(x)=x^2-1$. Notice that the function $f$ is real when the odd variables $\xi_1,\eta_1$ are set to $0$. Let $\Theta(x)$ be the real Heaviside function, i.e. $\Theta(x)=1$ when $x\geq 0$ and $\Theta(x)=0$ when $x<0$. The set $\{R_1^2\leq 1\}$ is then defined formally by the condition that $f\leq 0$; more precisely, we define
\begin{equation*}
    \int_{\{R_1^2\leq 1\}}\sigma=\int_{\mathbb{R}^{3n|2n}}\Theta(-f)\sigma
\end{equation*}
for any integral form $\sigma$ on $\mathbb{R}^{3n|2n}$ of codimension $0$ such that $\Theta(-f)\sigma$ has compact support. It should be noted that in light of~\eqref{eq:BasicTaylor}, we will use formulas such as
\begin{equation*}
    \Theta(a+b\xi\eta)=\Theta(a)+b\xi\eta\delta(a)
\end{equation*}
where $a$ is real, as in the discussion below~\cite[Equation~(3.54)]{Witten_2019}.

\subsection{Dimensional reduction}\label{sub:dimensional_reduction}
We now give a supersymmetric interpretation to~\cite[Equation~(5) and Lemma~8]{Lata_a_2017} or~\cite[Lemma~2.2]{Barthe2019}, where Royen constructed a special probability distribution to identify with Laplace transform calculations. Indeed, we show by supersymmetric dimensional reduction that the probability of $\{X_i^2\}\in[0,1]^{n}$ is the same as the probability of $\{R_i^2\}\in U_n$ for the super-radius. This amounts to showing that the extra bosons $Y,Z$ exactly cancel with the extra fermions $\xi,\eta$: this is known as (Parisi-Sourlas) supersymmetric dimensional reduction~\cite{Parisi_1979}. The special probability distribution that Royen constructed will then be shown to be the law of the bosonic body part of $R_i^2$.

We introduce as in~\cite{MR2728731} the distinguished supersymmetric operator $Q$ that interchanges $(Y,Z)$ with $(\xi,\eta)$ as $Q=\sum_{i=1}^{n}Q_i$ with
\begin{equation}\label{eq:Q}
    Q_i=\xi_i\frac{\partial}{\partial Y_i}+\eta_i\frac{\partial}{\partial Z_i}+Y_i\frac{\partial}{\partial\eta_i}-Z_i\frac{\partial}{\partial\xi_i}.
\end{equation}
Since the combination $Y_i^2+Z_i^2+2\xi_i\eta_i$ is $Q_i$-closed, the action $S_{\text{tot}}(X,Y,Z,\xi,\eta)$ is also $Q$-closed, and the flat Berezin volume form $d\mu$ is $Q$-invariant. The supersymmetric localization formula~\cite[Appendix~C]{MR2728731}) yields the following:
\begin{lemm}[Supersymmetric dimensional reduction from $\mathbb{R}^{3|2}$ to $\mathbb{R}^{1|0}$]\label{lemm:dim_red}
For any $F:\mathbb{R}^{n}\to\mathbb{R}$ smooth with enough decay that the integrals exist,
\begin{equation*}
    \int_{(\mathbb{R}^{1|0})^{n}}e^{-S_{\text{b}}(X)}F(X_1^2,\dots,X_n^2)=\int_{(\mathbb{R}^{3|2})^{n}}e^{-S_{\text{tot}}(X,Y,Z,\xi,\eta)}F(R_1^2,\dots,R_n^2),
\end{equation*}
where we omitted the respective flat Berezin forms in the integrals.
\end{lemm}
The proof relies on the fundamental supersymmetric localization theorem and is included in Appendix~\ref{sub:dimensional_reduction} for completeness. In the situation we are interested in, the Heaviside function $\Theta$ as a real function can be approximated by smooth and compactly supported functions. In particular, the above Lemma~\ref{lemm:dim_red} yields that for any $\emptyset\neq J\subset\{1,\dots,n\}$,
\begin{equation*}
    \mathbb{P}\left[\max_{i\in J}X_i^2\leq 1\right]=\mathbb{P}\left[\max_{i\in J}R_i^2\leq 1\right],
\end{equation*}
where the probability on the right hand side should be understood in the sense of supersymmetric integration (normalized by the partition function $Z=\det(C)^{\frac{1}{2}}$) as in Lemma~\ref{lemm:dim_red}. Therefore, the statistics of the super-radius $R_i$ is reduced to that of the real radius $|X_i|$, concluding the exactness of the supersymmetric lift from $\mathbb{R}^{1|0}$ to $\mathbb{R}^{3|2}$.

The situation is similar to the supersymmetric dimensional lift in the supersymmetric hyperbolic sigma models from $H^{0|2}$ to $H^{2|4}$, see~\cite[Proposition~2.7]{MR4218682} where we learned the use of the dimensional reduction technique.

\subsection{Supergeometry behind the Gaussian correlation inequality}\label{sub:supergeometry_behind_the_gaussian_correlation_inequality}
Define the bosonic body part of the squared super-radius $R_i^2$ by $B_i^2=X_i^2+Y_i^2+Z_i^2$. By Taylor expansion~\eqref{eq:BasicTaylor},
\begin{equation}\label{eq:Taylor_RB}
    \mathbf{1}_{\{R_i^2\leq 1\}}=\mathbf{1}_{\{B_i^2\leq 1\}}-2\xi_i\eta_i\delta(1-B_i^2).
\end{equation}
The supersymmetric picture provides an alternative approach to Royen's Laplace transform argument. In particular, the desired probability distribution $h_{3,C}$ constructed in~\cite[Lemma~8]{Lata_a_2017} or~\cite[Lemma~2.2]{Barthe2019} appears exactly as the law of $\frac{1}{2}B_i^2$ above, without any delicate Laplace transform identification. In other words, there is no need to manually find a well-designed explicit probability distribution to fit the Laplace transform; the supersymmetric dimensional reduction naturally constructs it for us.

We now show how the probability distribution in Royen's proof appears, before presenting the final proof. By dimensional reduction, we study the sign of
\begin{equation*}
    \frac{\partial}{\partial\tau}\mathbb{P}\left[\max_{1\leq i\leq n}X_i^2\leq 1\right]=\frac{\partial}{\partial\tau}\mathbb{P}\left[\max_{1\leq i\leq n}R_i^2\leq 1\right]=\frac{\partial}{\partial\tau}\left[\det(C)^{-\frac{1}{2}}\int_{(\mathbb{R}^{3|2})^{n}} d\mu\, e^{-S_{\text{tot}}(X,Y,Z,\xi,\eta)}\mathbf{1}_{U_n}\right],
\end{equation*}
where the $\tau$-dependence of the variables is implicit. The indicator over $U_n$ is the product of the above expansion~\eqref{eq:Taylor_RB} over $1\leq i\leq n$: this can be organized as
\begin{equation*}
    \sum_{J\subset\{1,\dots,n\}}\prod_{j\in J}(-2\xi_j\eta_j\delta(1-B_j^2))\prod_{i\in J'}\mathbf{1}_{\{B_i^2\leq 1\}}
\end{equation*}
where $J'=\{1,\dots,n\}\setminus J$. Therefore
\begin{equation}\label{eq:J_expansion}
    \int_{(\mathbb{R}^{3|2})^{n}} d\mu\, e^{-S_{\text{tot}}}\mathbf{1}_{U_n}=\sum_{J\subset\{1,\dots,n\}}\int_{(\mathbb{R}^{3|2})^{n}} d\mu\, e^{-S_{\text{b}}-S_{\text{f}}}\prod_{j\in J}(-2\xi_j\eta_j\delta(1-B_j^2))\prod_{i\in J'}\mathbf{1}_{\{B_i^2\leq 1\}}.
\end{equation}
Similarly to Section~\ref{sub:a_determinantal_expansion_formula}, we first integrate out the fermionic variables $\prod_{j\in J}\xi_j\eta_j$. Since the fermionic variables and the bosonic variables decouple in the action, the fermionic Wick contraction yields that the normalized fermionic expectation is
\begin{equation}\label{eq:FermionWick}
    \mathbb{E}_{\text{f}}\left[\prod_{j\in J}-2\xi_j\eta_j\right]=\det(C)\int_{(\mathbb{R}^{3|2})^{n}} d\mu_{\text{f}}\, e^{-S_{\text{f}}}\prod_{j\in J}(-2\xi_j\eta_j)=2^{|J|}\det(C_{J}).
\end{equation}
Notice the appearance of $\det(C_{J})$: if we hit it with $\frac{\partial}{\partial\tau}$, this gives $-a_J(\tau)$ which we studied in Lemma~\ref{lemm:AirJordan}. The remaining bosonic expectation is
\begin{equation}\label{eq:BosonicPositive}
    \det(C)^{-\frac{3}{2}}\int_{(\mathbb{R}^{3|0})^n}e^{-S_b(X,Y,Z;\tau)}\prod_{j\in J}\delta(1-B_j^2)\prod_{i\in J'}\mathbf{1}_{\{B_i^2\leq 1\}}.
\end{equation}
All we need for the proof is that this integral is \emph{positive}. If we denote by $h_{\tau}(x_1,\dots,x_n)$ the density of the random variable $(B_1(\tau)^2,\dots,B_n(\tau)^2)$ (this is exactly the random variable Royen used to construct the special probability density $h_{3,C}$), the last expression integrates exactly to
\begin{equation*}
    \int_{\prod_{i\in J'}[0,1]}h_{\tau}(1_{J},x_{J'})dx_{J'}.
\end{equation*}
This quantity matches with the last expression in~\cite[Section~2]{Lata_a_2017}, which is the key element in the final step in Royen's original proof.

However, these observations do not yield directly the result if one compares with Royen's calculations: we are missing a multiplicative factor $-\frac{1}{2}$, and the $\tau$-derivative should also hit the bosonic expectation term (as well as the partition function renormalization). But this is exactly where the supersymmetric localization method comes to the rescue. Intuitively, the variation of each of the bosonic component is exactly half of the fermionic variation generated by the pair of fermions, and since we have $3$ bosons and $2$ fermions, the total variation is the fermionic variation multiplied by the factor $-3\cdot \frac{1}{2}+1=-\frac{1}{2}$.

\subsection{Supersymmetric localization and proof of the Gaussian correlation inequality}
We are now in a position to write down the supersymmetric translation of Royen's proof. We first gather some notations and basic facts. Let $n=n_1+n_2$ as in Theorem~\ref{th:main}.
\begin{itemize}
    \item The covariance matrix $C(\tau)$ is defined in~\eqref{eq:Interpolate_Cov}, with $\tau\in[0,1]$.
    \item The bosonic action $S_{\text{b}}(X;\tau)$ is defined in~\eqref{eq:Action_X}; its dependence on $\tau$ is now written.
    \item The fermionic action $S_{\text{f}}(\xi,\eta;\tau)=\sum_{1\leq i,j\leq n}\xi_i (C(\tau)^{-1})_{ij}\eta_j$ is also defined above.
    \item The supersymmetric operator $Q^{Y,Z}=\sum_{1\leq i\leq n}Q^{Y,Z}_i$ with $Q^{Y,Z}_i$ in~\eqref{eq:Q} interchanges $(Y,Z)$ with $(\xi,\eta)$; we will also consider its permutations $Q^{X,Y}$ and $Q^{Z,X}$.
    \item The super-radius $R_{i}=\sqrt{X_i^2+Y_i^2+Z_i^2+2\xi_i\eta_i}$ is invariant under all $Q$ above.
    \item The body part of $R_{i}$, namely $B_{i}=\sqrt{X_i^2+Y_i^2+Z_i^2}$, is viewed as an element of $\mathbb{R}_{\geq 0}$.
\end{itemize}

As explained in Section~\ref{sub:dimensional_reduction} using dimensional reduction, we study the $\tau$-variation of
\begin{equation}\label{eq:total_tau_derivative}
\begin{split}
    \frac{\partial}{\partial\tau}\mathbb{P}\left[\max_{1\leq i\leq n}X_i^2\leq 1\right]&=\frac{\partial}{\partial\tau}\mathbb{P}\left[\max_{1\leq i\leq n}R_i^2\leq 1\right]\\
    =&\frac{\partial}{\partial\tau}\left[\det(C(\tau))^{-\frac{1}{2}}\int_{(\mathbb{R}^{3|2})^{n}} d\mu\, e^{-S_{\text{tot}}(X,Y,Z,\xi,\eta;\tau)}\mathbf{1}_{U_n}\right]
\end{split}
\end{equation}
and show that it is positive for $\tau\in[0,1]$.
\begin{lemm}[Ward identity or switching lemma]\label{lemm:BosonFermion}
The $\tau$-derivative above is equal to
\begin{equation}\label{eq:tau'_tau}
    \frac{\partial}{\partial\tau}\left[-\frac{1}{2}\det(C(\tau'))^{-\frac{3}{2}}\det(C(\tau))\int_{(\mathbb{R}^{3|2})^{n}} d\mu\, e^{-S_{\text{b}}(X,Y,Z;\tau')}e^{-S_{\text{f}}(\xi,\eta;\tau)}\mathbf{1}_{U_n}\right]_{\big|\tau'=\tau},
\end{equation}
where the subscript means that we first treat $\tau,\tau'$ as independent variables in the expression inside $[\cdot]$, take the partial derivative with respect to $\tau$ only, and finally evaluate at $\tau'=\tau$. In other words, the variation~\eqref{eq:total_tau_derivative} is proportional to the variation with respect only to the fermionic variables, multiplied by the constant $-\frac{1}{2}$.
\end{lemm}
\begin{proof}
We study the $\tau$-variation of $\frac{\partial}{\partial\tau}\left[\det(C(\tau))^{-\frac{1}{2}}\int_{(\mathbb{R}^{3|2})^{n}} d\mu\, e^{-S_{\text{tot}}(X,Y,Z,\xi,\eta;\tau)}\mathbf{1}_{U_n}\right]$.

Let us start with the partition function $\det(C(\tau))^{\frac{1}{2}}$ as an illustration. By definition,
\begin{equation*}
    \det(C(\tau))^{\frac{1}{2}}=\int d\mu\, e^{-S_{\text{tot}}(X,Y,Z,\xi,\eta;\tau)}.
\end{equation*}
Differentiating with respect to $\tau$ and using the symmetry between $X,Y,Z$ yields
\begin{equation*}
    \frac{\partial}{\partial\tau}\det(C(\tau))^{\frac{1}{2}}=-\int d\mu\, e^{-S_{\text{tot}}(X,Y,Z,\xi,\eta;\tau)}\left(\frac{3}{2}Y^{t}\frac{\partial}{\partial\tau}(C(\tau)^{-1})Y+\xi^{t}\frac{\partial}{\partial\tau}(C(\tau)^{-1})\eta\right).
\end{equation*}
Notice that $Q^{Y,Z}(Y^{t}\Sigma\eta)=\xi^{t}\Sigma\eta+Y^{t}\Sigma Y$ for any symmetric matrix $\Sigma$. Therefore, we can perform the supersymmetric $Q^{Y,Z}$ integration by parts using the supersymmetric localization formula Lemma~\ref{lemm:SusyLoca} to get
\begin{equation*}
\begin{split}
    \frac{\partial}{\partial\tau}\det(C(\tau))^{\frac{1}{2}}&=-\int d\mu\, e^{-S_{\text{tot}}(X,Y,Z,\xi,\eta;\tau)}\left(\frac{3}{2}Y^{t}\frac{\partial}{\partial\tau}(C(\tau)^{-1})Y+\xi^{t}\frac{\partial}{\partial\tau}(C(\tau)^{-1})\eta\right)\\
    &=-\int d\mu\, e^{-S_{\text{tot}}(X,Y,Z,\xi,\eta;\tau)}\left(Q(\dots)-\frac{3}{2}\xi^{t}\frac{\partial}{\partial\tau}(C(\tau)^{-1})\eta+\xi^{t}\frac{\partial}{\partial\tau}(C(\tau)^{-1})\eta\right)\\
    &=\frac{1}{2}\int d\mu\, e^{-S_{\text{tot}}(X,Y,Z,\xi,\eta;\tau)}\xi^{t}\frac{\partial}{\partial\tau}(C(\tau)^{-1})\eta\\
    &=-\frac{1}{2}\frac{\partial}{\partial\tau}\left[\int d\mu\, e^{-S_{\text{b}}(X,Y,Z;\tau')}e^{-S_{\text{f}}(\xi,\eta;\tau)}\right]_{\big|\tau'=\tau}.
\end{split}
\end{equation*}

We now move to the integral. Since each $R_i$ is $Q$-invariant, the extra indicator $\mathbf{1}_{U_n}$ is also $Q$-invariant. The argument above goes through verbatim and yields
\begin{equation*}
    \frac{\partial}{\partial\tau}\left[\int d\mu\, e^{-S_{\text{tot}}(X,Y,Z,\xi,\eta;\tau)}\mathbf{1}_{U_n}\right]=-\frac{1}{2}\frac{\partial}{\partial\tau}\left[\int d\mu\, e^{-S_{\text{b}}(X,Y,Z;\tau')}e^{-S_{\text{f}}(\xi,\eta;\tau)}\mathbf{1}_{U_n}\right]_{\big|\tau'=\tau}.
\end{equation*}

Together with the chain rule, we get
\begin{equation*}
\begin{split}
    &\frac{\partial}{\partial\tau}\left[\det(C(\tau))^{-\frac{1}{2}}\int_{(\mathbb{R}^{3|2})^{n}}d\mu\, e^{-S_{\text{tot}}(X,Y,Z,\xi,\eta;\tau)}\mathbf{1}_{U_n}\right]\\
    ={}&\left(\frac{\partial}{\partial\tau}\det(C(\tau))^{-\frac{1}{2}}\right)\int d\mu\, e^{-S_{\text{tot}}}\mathbf{1}_{U_n}+\det(C(\tau))^{-\frac{1}{2}}\frac{\partial}{\partial\tau}\left[\int d\mu\, e^{-S_{\text{tot}}(X,Y,Z,\xi,\eta;\tau)}\mathbf{1}_{U_n}\right]\\
    ={}&\left(-\frac{1}{2}\det(C(\tau))^{-\frac{3}{2}}\frac{\partial}{\partial\tau}\det(C(\tau))\right)\left[\int d\mu\, e^{-S_{\text{b}}(X,Y,Z;\tau')}e^{-S_{\text{f}}(\xi,\eta;\tau)}\mathbf{1}_{U_n}\right]_{\big|\tau'=\tau}\\
    &\quad -\frac{1}{2}\det(C(\tau))^{-\frac{1}{2}}\frac{\partial}{\partial\tau}\left[\int d\mu\, e^{-S_{\text{b}}(X,Y,Z;\tau')}e^{-S_{\text{f}}(\xi,\eta;\tau)}\mathbf{1}_{U_n}\right]_{\big|\tau'=\tau}\\
    ={}&-\frac{1}{2}\det(C(\tau))^{-\frac{3}{2}} \frac{\partial}{\partial\tau}\left[\det(C(\tau))\int_{(\mathbb{R}^{3|2})^{n}}d\mu\, e^{-S_{\text{b}}(X,Y,Z;\tau')}e^{-S_{\text{f}}(\xi,\eta;\tau)}\mathbf{1}_{U_n}\right]_{\big|\tau'=\tau},
\end{split}
\end{equation*}
which is the desired identity.
\end{proof}

The connection we have unveiled so far is summarized as follows:
\begin{equation*}
\begin{array}{|c|c|}
\hline
\text{Royen's proof} & \text{Supersymmetric interpretation}\\
\hline
\text{Auxiliary Gamma law } h_{3,C}
& \text{Body of the super-radius on } \mathbb R^{3|2}\\
\text{Principal-minor expansion}
& \text{Fermionic expansion}\\
\text{Derivative of determinants $a_{J}(\tau)$}
& \text{Proportional to fermionic variations}\\
\text{Multiplicative factor } -1/2
& \text{Ward identity: } 3\text{ bosons and }2\text{ fermions}\\
\text{Boundary integrals}
& \text{Expansion of the super-hypercube indicator}\\
\hline
\end{array}
\end{equation*}

The supersymmetric proof of the Gaussian correlation inequality is now ready.
\begin{proof}[Proof of Theorem~\ref{th:main}]
By Lemma~\ref{lemm:BosonFermion}, we only need to show that~\eqref{eq:tau'_tau} is positive for $\tau\in[0,1]$. Expand the integral in~\eqref{eq:tau'_tau} as in~\eqref{eq:J_expansion}, the bosonic part (with $\tau'$) is independent of $\tau$ and is factorized out, leaving only the $\tau$-variation of the fermionic part, which evaluates to
\begin{equation*}
    2^{|J|}\frac{\partial}{\partial\tau}\det(C_J(\tau))=-2^{|J|}a_J(\tau)\leq 0
\end{equation*}
by~\eqref{eq:FermionWick} and Lemma~\ref{lemm:AirJordan}. For each $J\subset\{1,\dots,n\}$, the bosonic part is positive by~\eqref{eq:BosonicPositive}. Combining the factor $-\frac{1}{2}$, the negative fermionic variation, and the positive bosonic part yields a positive final result.
\end{proof}

\section{Improved general correlation inequality}\label{sec:the_case_of_the_multivariate_gamma_distributions}
In the supersymmetric framework, the extension of the Gaussian correlation inequality to the case of the half-integer multivariate Gamma distributions requires no new machinery: it is the same dimensional reduction except starting from a higher initial bosonic dimension. It can also be seen from the previous discussion that the Gaussian Correlation Inequality follows as a special case of some broader class of inequalities via supersymmetric localization, and the more general result stated in Theorem~\ref{th:better} will be proved in this section.

Let us recall the setup of Theorem~\ref{th:better} with the notation of Section~\ref{sec:supersymmetric_reinterpretation_of_royen_s_proof}. Let $k\geq 1$ be an integer and consider $k$ independent copies of the previous Gaussian vector $X$ with covariance matrix $C(\tau)$, denoted by $X^{(1)},\dots, X^{(k)}$. The multivariate Gamma vector with parameter $\frac{k}{2}$, denoted by $\Gamma_{n}(\frac{k}{2},C(\tau))$, is introduced in Theorem~\ref{th:better} and will be abbreviated as $\Gamma$: the $\tau$ dependence is implicit in the sequel. It is well-known that $\Gamma=(\Gamma_1,\dots,\Gamma_n)$ can be defined component-wise as a sum of squares
\begin{equation*}
    \Gamma_i=\frac{1}{2}\sum_{m=1}^{k}(X^{(m)}_i)^2.
\end{equation*}
When $k=1$, we recover the setting of the Gaussian case of Section~\ref{sec:supersymmetric_reinterpretation_of_royen_s_proof}.

\subsection{Dimensional lift}
For $k=1$ this construction reduces to the Gaussian case discussed above. For general $k\geq 1$, we perform the same dimensional lift from $\mathbb{R}^{k|0}$ to $\mathbb{R}^{k+2|2}$ as in Section~\ref{sub:dimensional_reduction}. That is, we introduce independent copies $Y$ and $Z$ of $X$, as well as their fermionic counterparts $(\xi,\eta)$ with the same covariance matrix $C(\tau)$. More precisely, define the bosonic action
\begin{equation*}
    S_{\text{b}}(X^{(1)},\dots,X^{(k)})=\frac{1}{2}\sum_{m=1}^{k}\sum_{1\leq i,j\leq n}X_i^{(m)}(C(\tau)^{-1})_{ij}X_j^{(m)}
\end{equation*}
and the total action (where $S_{\text{b}}(Y),S_{\text{b}}(Z)$ are defined as in~\eqref{eq:Action_X})
\begin{equation*}
    S_{\text{tot}}=S_{\text{b}}(X^{(1)},\dots,X^{(k)})+S_{\text{b}}(Y)+S_{\text{b}}(Z)+\sum_{1\leq i,j\leq n}\xi_i(C(\tau)^{-1})_{ij}\eta_j.
\end{equation*}
The extended super-radius becomes
\begin{equation*}
    R_i=\sqrt{\sum_{m=1}^{k}(X^{(m)}_i)^2+Y_i^2+Z_i^2+2\xi_i\eta_i},
\end{equation*}
whose bosonic body part is defined as
\begin{equation*}
    B_i=\sqrt{\sum_{m=1}^{k}(X^{(m)}_i)^2+Y_i^2+Z_i^2}.
\end{equation*}
Analogous dimensional reduction with the same operator $Q=\sum_{i=1}^{n}Q_i$ as in~\eqref{eq:Q} yields in particular that for any smooth test function $\Phi$ with compact support or sufficient decay,
\begin{equation*}
    \mathbb{E}\left[\Phi\left((2\Gamma_i)_{i\in[1,\dots,n]}\right)\right]=\mathbb{E}\left[\Phi\left((R_i^2)_{i\in[1,\dots,n]}\right)\right].
\end{equation*}

\begin{rema}
Since $\frac{1}{2}B_i^2$ is a sum of squares of independent Gaussian variables, it has a multivariate Gamma distribution of parameter $\frac{k+2}{2}$. Therefore its Laplace transform is explicit (as recalled in Theorem~\ref{th:better}) and is used crucially in Royen's proof. In contrast, we will not explicitly use the exact law of $B_i$ nor its Laplace transform in the sequel.
\end{rema}

\subsection{Extension of the previous supersymmetric proof}
Consider the $\tau$-derivative of
\begin{equation*}
    \mathbb{E}\left[\Phi((R_i^2)_{i\in[1,\dots,n]})\right]=\det(C)^{-\frac{k}{2}}\int_{(\mathbb{R}^{k+2|2})^{n}} d\mu\, e^{-S_{\text{tot}}}\Phi((R_i^2)_{i\in[1,\dots,n]}).
\end{equation*}
By examining the proof of Lemma~\ref{lemm:BosonFermion} (essentially by replacing the indicator function by the general test function $\Phi$), we get that
\begin{equation}\label{eq:Gamma_Lemma}
\begin{split}
    &\frac{\partial}{\partial\tau}\left[\det(C)^{-\frac{k}{2}}\int_{(\mathbb{R}^{k+2|2})^{n}} d\mu\, e^{-S_{\text{tot}}}\Phi((R_i^2)_{i\in[1,\dots,n]})\right]\\
    ={}&-\frac{k}{2}\frac{\partial}{\partial\tau}\left[\det(C(\tau'))^{-\frac{k+2}{2}}\det(C(\tau))\int_{(\mathbb{R}^{k+2|2})^{n}} d\mu\, e^{-S_{\text{b}}(\tau')-S_{\text{f}}(\tau)}\Phi((R_i^2)_{i\in[1,\dots,n]})\right]_{\big|\tau'=\tau},
\end{split}
\end{equation}
where $S_{\text{b}}(\tau)=S_{\text{b}}(X^{(1)},\dots,X^{(k)};\tau)+S_{\text{b}}(Y;\tau)+S_{\text{b}}(Z;\tau)$. This is because the $\tau$ derivatives only act on the partition function and the action term, therefore the same boson-fermion cancellation as in Lemma~\ref{lemm:BosonFermion} appears here.

We now Taylor expand the expression
\begin{equation*}
    \Phi((R_i^2)_{i\in[1,\dots,n]})=\sum_{J\subset[1,\dots,n]}\partial_J\Phi((B_i^2)_{i\in[1,\dots,n]})2^{|J|}\prod_{j\in J}\xi_j\eta_j.
\end{equation*}
Similarly to~\eqref{eq:J_expansion}, with the flat Berezin volume form on respective integration domains:
\begin{equation*}
\begin{split}
    &\int d\mu\, e^{-S_{\text{b}}(\tau')-S_{\text{f}}(\tau)}\Phi((R_i^2)_{i\in[1,\dots,n]})\\
    ={}&\sum_{J\subset\{1,\dots,n\}}2^{|J|}\int d\mu\, e^{-S_{\text{b}}(\tau')-S_{\text{f}}(\tau)}\left(\partial_J\Phi((B_i^2)_{i\in[1,\dots,n]})\prod_{j\in J}\xi_j\eta_j\right).
\end{split}
\end{equation*}
Therefore, factoring out the positive bosonic factor (whose exact law can be ignored if we only care about the final sign), we are left again with the $\tau$-derivative of the fermionic expectation contribution consisting of the $\xi,\eta$ terms. Fermionic integration is done as in the Gaussian case, yielding exactly $-(-2)^{|J|}a_{J}(\tau)$ using the fermionic Wick contraction and Lemma~\ref{lemm:AirJordan}. Multiplying by the constant $-\frac{k}{2}$ in~\eqref{eq:Gamma_Lemma}, this concludes the proof of the formula in Theorem~\ref{th:better}.

\begin{rema}
In particular, if $(-1)^{|J|}\partial_J\Phi$ is non-negative for all $J\subset[1,\dots,n]$ on $\mathbb{R}_{\geq 0}^{n}$, we have $\frac{\partial}{\partial\tau}\mathbb{E}\left[\Phi((\Gamma_i)_{i\in[1,\dots,n]})\right]\geq 0$ as in the Gaussian case.
\end{rema}

\appendix
\section{Proof of the dimensional reduction lemma}\label{sec:proof_of_the_dimensional_reduction_lemma}
We sketch the proof of Lemma~\ref{lemm:dim_red} following~\cite[Proposition~2.7]{MR4218682}: it should be noticed that the roles of $\xi,\eta$ are reversed therein, but the proofs are essentially the same.

Recall the supersymmetric localization theorem~\cite[Lemma~16]{MR2728731} (stated here in the flat geometry):
\begin{lemmA}\label{lemm:SusyLoca}
Let $f=f(Y,Z,\xi,\eta)$ be a superfunction acting on even elements, satisfying the relation $Qf=0$ with $Q$ defined in~\eqref{eq:Q}, and decays sufficiently fast for the integral $\int_{(\mathbb{R}^{2|2})^{n}} d\mu\, f$ to exist, where $d\mu$ is the flat Berezin volume form. Then
\begin{equation*}
    \int d\mu\, f=f(0,0,0,0).
\end{equation*}
\end{lemmA}
We apply this to the right-hand side of the equation in Lemma~\ref{lemm:dim_red}. We fix the $X$ component and integrate over $(Y,Z,\xi,\eta)\in(\mathbb{R}^{2|2})^n$. Then since $Q(Y_i^2+Z_i^2+2\xi_i\eta_i)=0$, we have $Q(S_{\text{b}}(Y,Z)+S_{\text{f}}(\xi,\eta))=0$ and applying the above lemma with $f(Y,Z,\xi,\eta)=e^{-S_{\text{b}}(X)}e^{-S_{\text{b}}(Y,Z)-S_{\text{f}}(\xi,\eta)}F(X,Y,Z,\xi,\eta)$ shows that $R_i^2$ localizes to $X_i^2$.

\bigskip

\subsection*{Conflict of interest}
The authors have no relevant financial or non-financial interests to disclose.

\subsection*{Data availability}
Data sharing is not applicable to this article as no new data were created or analyzed in this work.

\bibliographystyle{alpha}

\end{document}